\documentclass[12pt]{article}
\usepackage{float}
\usepackage{makeidx}
\usepackage{latexsym}

\usepackage[margin=2.5cm]{geometry}
\usepackage[dvips]{graphicx}
\usepackage[english]{babel}
\usepackage{amssymb}
\usepackage{amsmath}
\usepackage{amsthm}
\usepackage{lineno}
\usepackage{xcolor}
\definecolor{blau}{rgb}{0.1,0.0,0.9}
\definecolor{gruen}{cmyk}{1.0,0.2,0.7,0.07}
\definecolor{mag}{cmyk}{0.0,0.9,0.3,0.0}

\DeclareMathOperator{\mad}{Mad}
\DeclareMathOperator{\ch}{ch}

\newtheorem{theorem}{Theorem}[section]
\newtheorem{lemma}[theorem]{Lemma}
\newtheorem{corollary}[theorem]{Corollary}
\newtheorem{proposition}[theorem]{Proposition}

\newtheorem{observation}[theorem]{Observation}
\newtheorem{remark}[theorem]{Remark}
\theoremstyle{definition}
\newtheorem{problem}[theorem]{Problem}

\begin{document}
\date{\today}
\title{A note on adaptable choosability and choosability with separation of planar graphs}

\author{
Carl Johan Casselgren\footnote{Department of Mathematics, 
Link\"oping University, 
SE-581 83 Link\"oping, Sweden.
{\it E-mail address:} carl.johan.casselgren@liu.se}
\and
Jonas B. Granholm\footnote{Department of Mathematics, 
Link\"oping University, 
SE-581 83 Link\"oping, Sweden.
{\it E-mail address:} jonas.granholm@liu.se}
\and
Andr\'e Raspaud\footnote{LaBRI, Uiversity of Bordeaux, France.
{\it E-mail address:} raspaud@labri.fr}
}
\maketitle

\bigskip
\noindent
{\bf Abstract.}
Let $F$ be a (possibly improper) edge-coloring of a graph $G$; a
vertex coloring of $G$ is \emph{adapted to} $F$ if no color appears
at the same time on an edge and on its two endpoints. If for some
integer $k$, a graph $G$ is such that given any list assignment $L$ to
the vertices of $G$, with $|L(v)| \ge k$ for all $v$, and any
edge-coloring $F$ of $G$, $G$ admits a coloring $c$ adapted to $F$
where $c(v) \in L(v)$ for all $v$, then $G$ is said to be
\emph{adaptably $k$-choosable}. 
A {\em $(k,d)$-list assignment} for a graph $G$ is a
map that assigns to each vertex $v$ a list $L(v)$ of at least $k$ colors
such that $|L(x) \cap L(y)| \leq d$ whenever $x$ and $y$ are adjacent.
A graph is {\em $(k,d)$-choosable} if 
for every $(k,d)$-list assignment $L$
there is an $L$-coloring of $G$.
It has been conjectured that planar graphs are $(3,1)$-choosable. 
We give some progress on this conjecture by giving  sufficient conditions for a planar
graph to be adaptably $3$-choosable. 
Since $(k,1)$-choosability is a special case of adaptable $k$-choosablity,
this implies
that a planar graph satisfying  these conditions is  $(3,1)$-choosable.

\bigskip


\section{Introduction}

Given a graph $G$, assign to each vertex $v$ of $G$ a set
	$L(v)$ of colors (positive integers).
	Such an assignment $L$ is called a \emph{list assignment} for $G$ and
	the sets $L(v)$ are referred
	to as \emph{lists} or \emph{color lists}.
	We then want to find a proper
	vertex coloring $\varphi$ of $G$,
	such that $\varphi(v) \in L(v)$ for all
	$v \in V(G)$. If such a coloring $\varphi$ exists then
	$G$ is \emph{$L$-colorable} and $\varphi$
	is called an \emph{$L$-coloring}. Furthermore, $G$ is
	called \emph{$k$-choosable} if it is $L$-colorable
	for every $k$-list assignment $L$.
	
	This particular variant of vertex coloring is
	known as \emph{list coloring} or {\em choosability} of graphs
	and was introduced by Vizing \cite{Vizing} and independently
	by Erd\H os et al. \cite{ERT}.
	
	A recent variation on list coloring is the so-called model
	of choosability with separation where we require that lists of adjacent
	vertices have a bounded number of common colors.
	A {\em $(k,d)$-list assignment} for a graph $G$ is a
	map that assigns to each vertex $v$ a list $L(v)$ of at least $k$ colors
	such that $|L(x) \cap L(y)| \leq d$ whenever $x$ and $y$ are adjacent.
	A graph is {\em $(k,d)$-choosable} if for every $(k,d)$-list assignment $L$
	of $G$
	there is an $L$-coloring of $G$.
	Note that $G$ is $(k,k)$-choosable if and only if $G$ is $k$-choosable.
	Moreover, if $G$
	is $(k,d)$-choosable, then $G$ is also $(k',d')$-choosable for all 
	$k' \geq k$ and $d' \leq d$.
	
	Choosability with separation was first considered by Kratochvil et al.
	\cite{KratochvilTuzaVoigt}. Among other things they proved
	that every planar graph is $(4,1)$-choosable, which is a refinement 
	for choosability of separation of
	Thomassen's well-known result that planar graphs are $5$-choosable
	\cite{Thomassen}.
	By an example of Voigt \cite{Voigt}, Thomassen's result is best possible.
	
	Skrekovski \cite{Skrekovski} gave examples of triangle-free planar graphs
	that are not $(3,2)$-choosable, and posed the following question: 
	
	\begin{problem}
	\label{ques:31choice}
		Is every planar graph $(3,1)$-choosable?
	\end{problem}
	
	It follows from a result of Kratochvil et al. \cite{KratochvilTuzaVoigt} 
	that this question has a positive answer for the case of 
	triangle-free graphs.
	Recently, Choi et al.~\cite{ChoiLidickyStolee}
	proved that planar graphs without $4$-cycles are 
	$(3, 1)$-choosable. 
	This was slightly improved by Chen et al.~\cite{ChenLihWang} 
	who proved that planar
	graphs with no adjacent $4$-cycles and no adjacent $3$- and $4$-cycles
	are $(3, 1)$-choosable, where two cycles of a graph are {\em adjacent}
	if they share a common edge; two cycles are {\em intersecting}
	if they have at least one common vertex.
	
	The main purpose of this note is to give some further progress on 
	Problem \ref{ques:31choice}. In particular we prove that
	a planar graph $G$ is $(3,1)$-choosable if $G$ satisfies that
	\begin{itemize}
	
		\item[(i)] no two triangles are intersecting, and every triangle
		is adjacent to at most one $4$-cycle, or
	
		\item[(ii)] no triangle is adjacent to a triangle or a $4$-cycle,
		and every $5$-cycle is adjacent to at most three triangles.
		
	\end{itemize}	
	Further related results on Problem \ref{ques:31choice} appear in
	\cite{ChoiLidickyStolee, ChenFanWangWang, KiersteadLidicky, ChenFanRasp}.\\

Let  $G$ be a graph  and let $F$ be a (possibly improper) coloring of the edges of $G$ with integers.
A $k$-coloring $c:V(G)\rightarrow \{1, \ldots, k\}$ of the vertices
of $G$ is \emph{adapted} to $F$ if for every $uv \in E(G)$, $c(u)
\neq c(v)$ or $c(v) \neq F(uv)$. In other words, there is no
monochromatic edge i.e.~an edge whose two ends are colored with the same
color as the edge itself. If there is an integer $k$ such that for
any edge coloring $F$ of $G$, there exists a vertex $k$-coloring
of $G$ adapted to $F$, we say that $G$ is \emph{adaptably
$k$-colorable}. The smallest $k$ such that $G$ is adaptably
$k$-colorable is called the \emph{adaptable chromatic number} of
$G$, denoted by $\chi_{ad}(G)$. The concept of adapted coloring of
graphs was introduced by Hell and Zhu in \cite{HZ07}.

Let $L$ 
be a list assignment for
the vertices of a graph $G$, and $F$ be a (possibly improper) edge
coloring of $G$. A coloring $c$ of $G$ adapted to $F$ is an
\emph{$L$-coloring adapted to} $F$ if for any vertex $v \in V(G)$,
we have $c(v) \in L(v)$. If for any edge coloring $F$ of $G$ and
any list assignment $L$ with $|L(v)|\ge k$ for all $v \in V(G)$
there exists an $L$-coloring of $G$ adapted to $F$, we say that $G$
is \emph{adaptably $k$-choosable}. The smallest $k$ such that $G$ is
adaptably $k$-choosable is called the \emph{adaptable choosability}
 (or the  \emph{adaptable choice number})  of $G$, denoted by
$\ch_{ad}(G)$. The concept of adapted list coloring of graphs and
hypergraphs was introduced by Kostochka and Zhu in \cite{KZ07}.\\

The following observation was first made in \cite{EsperetKangThomasse}.

\begin{observation}\label{OBS}
If $G$ is
adaptably $k$-choosable, then $G$ is $(k,
1)$-choosable.
\end{observation}

\begin{proof}
Assume that $G$ is adaptably $k$-choosable. 
Let $L$ be a $(k,1)$-list assignment for $G$.  
For any edge $e=xy$ of $G$ 
we color $xy$ with the unique element in $L(x) \cap L(y)$; 
let $F$ be this edge coloring of $G$. Since $G$ is adaptably $k$-choosable, 
there is a coloring of $G$ from the lists which is adapted to $F$.
Since any two adjacent vertices of $G$ have at most one common color
in their lists, this coloring is proper.
Hence, $G$ is $(k,1)$-choosable.
\end{proof}

Our results on $(3,1)$-choosability are based on this connection; thus
any planar graph satisfying (i) or (ii) is adaptably $3$-choosable.
In Section 2 we give some further connections between 
adaptable choosability and $(3,1)$-choosability based
on the edge-arboricity of a graph.

	Our main results are proved in Section 3.
	Our proofs are based on connections between the maximum average degree of a graph 
	and orientations of the underlying graph. 
	A benefit of our method is that it yields rather short proofs;
	many results in this area are based on rather lengthy discharging arguments,
	or uses precoloring extension techniques based on the proof of Thomassen's
	celebrated theorem on $5$-choosability of planar graphs \cite{Thomassen}
	combined with a detailed structural analysis of the graph
	(cf. \cite{ChoiLidickyStolee,ChenFanWangWang,ChenLihWang, ChenFanRasp}).
	
	In Section 4 we 
	note that yet another family of planar
	graphs are $(3,1)$-choosable, namely the so-called Halin graphs.
	
\section{Edge arboricity and Adaptable choosability}

	The {\em edge-arboricity} $a(G)$ of a graph $G$
	is the minimum number of forests into which its edges can be partitioned.
	It is well-known that if a graph has arboricity at most $d$, then
	it has an orientation with out-degree at most $d$ (see e.g. \cite{Eppstein}).

	The following proposition demonstrates the connection between
	adaptable choosability (and thus $(k,1)$-choosability) and
	orientations.
	
	 \begin{proposition}
	\label{prop:arboricity}	
		If $a(G)\le k$, then $G$ is adaptably $(k+1)$-choosable, and thus
		$(k+1,1)$-choosable.
	\end{proposition}

	\begin{proof}
	By assumption, $G$ has an orientation in which each vertex $x_i$
has $d^+(x_i) \le k$. Assume each vertex
$x_i$ is given a list $L(x_i)$ of $ k+1$ colors and $F$ is an edge coloring of $G$. Let $c(x_i)$
be any color in $L(x_i)$ which does not appear on any outgoing
edges of $x_i$. Then it is obvious that $c$ is an $L$-coloring of
$G$ adapted to $F$. This completes the proof of Proposition \ref{prop:arboricity}.
\end{proof}

	Since the edge-arboricity of a triangle-free planar graph is at most $2$ and the edge-arboricity of planar graphs is at most $3$, the preceding proposition yields yet another immediate proof of the
	facts that every triangle-free planar graph is $(3,1)$-choosable, and  that every planar graph is $(4,1)$-choosable \cite{KratochvilTuzaVoigt}.
	
	As pointed out above, Choi et al.~\cite{ChoiLidickyStolee}
	proved that planar graphs without $4$-cycles are 
	$(3,1)$-choosable, but there are  planar graphs without $4$-cycles which are not 
	adaptably $3$-colorable \cite{EMZ07}. 
	
However, as follows from Proposition \ref{prop:arboricity}, every
planar graph is adaptably $4$-choosable.
	We note that this in fact holds for any graph with no $K_5$-minor;
	which was first established in
	in \cite{EMZ07}.

\begin{corollary}
\label{cor2} Every $K_5$-minor free (simple) graph is adaptably
$4$-choosable.
\end{corollary}

As noted   in \cite{MontassierRaspaudZhu}, it is easy to prove that the 
edge-arboricity  of a $K_5$-minor free (simple) graph $G$ is at
most $3$ (this follows since such a graph satisfies $|E(G)| \leq 3|V (G)| - 6$);
so Proposition \ref{prop:arboricity} implies Corollary \ref{cor2}.
The latter statement yields the following.

\begin{corollary}
		If $G$ is a $K_5$-minor free graph, then $G$ is $(4,1)$-choosable.
	\end{corollary}

	\section{Sufficient conditions for adaptable $3$-choosability and 
	$(3,1)$-choosability of planar graphs} 
	
	In this section we prove our main results on adaptable 3-choosability and  
	$(3,1)$-choosability 
	of planar graphs.

	Given a graph $G$, the {\em maximum average degree} 
	of $G$, denoted by $\mad(G)$, 
	is the maximum of the average degrees of all subgraphs of $G$, i.e.,
	$$\mad(G) = \max\{2|E(H)|/|V (H)| : H \text{ is a subgraph of } G\}.$$
	
	We denote by $d(v)$ the degree of a vertex $v$ in $G$ and by $r(f)$
	the degree of a face $f$, i.e. the number of edges incident with it.
	A $k$-face is a face of degree $k$, and a $k^+$-face (respectively, $k^-$-face)
	is a face with degree at least $k$ (respectively, at most $k$).

	In \cite{MontassierRaspaudZhu}
the following theorem  is proved  using the orientation method 
presented in Section 2 (see \cite{HUBERT94} for a detailed introduction 
to orientation and Mad).

\begin{theorem}
\label{main} For any graph $G$,
$$\ch_{ad}(G) \le \lceil \mad(G)/2 \rceil + 1.$$
\end{theorem}

By using Obervation \ref{OBS} we have the following.

\begin{lemma}
\label{lem:orient}
Every graph $G$ is $(\lceil \mad(G)/2 \rceil + 1,1)$-choosable.
	
\end{lemma}

Now we prove the following theorem.
\begin{theorem}
\label{th:planar1}
If $G$ is a planar graph with no intersecting 
triangles and where every triangle is adjacent to at most
one $4$-cycle, then $\mad(G) < 4$.
\end{theorem}

\begin{proof}
Suppose, for a contradiction, that there is a
subgraph $H$ of $G$ with average degree at least $4$. Clearly,
we may assume that $H$ is connected.
We shall use a discharging argument for proving that $H$
has average degree less than $4$, thus yielding the desired contradiction.

Let $V, E,F$ be the sets of vertices, edges and faces of $H$, respectively.
By Euler's formula $|V|-|E| + |F|=2$, we have
\begin{equation}
\nonumber
4|E|-4|V| -4|F|=-8.
\end{equation}
Rewriting this yields
\begin{equation}
\label{eq:euler}
\sum_{v \in V}(d(v)-4) + \sum_{f\in F}(r(f)-4)=-8.
\end{equation}

We now define a weight function $\omega : V \cup F \to \mathbb{R}$
by setting $w(v) = d(v)-4$ if $v\in V$, and
$w(f) = r(f)-4$ if $f \in F$.

A $3$-face that shares two edges with an adjacent $5$-face is
called a {\em bad} $3$-face for this $5$-face; a $3$-face
that shares one edge with an adjacent $5$-face is called an
{\em ordinary} $3$-face for this $5$-face.
Note that a $5$-face has at most one bad $3$-face.

Our discharging procedure is simple.

\begin{itemize}

	\item [(R1)] A $5$-face gives $\frac{1}{2}$ to each adjacent
	$3$-face if none of them are bad; if one adjacent $3$-face is bad,
	then the $5$-face gives $1$ to this bad $3$-face, and nothing to any
	other adjacent ordinary $3$-face.
	
	\item[(R2)]
	A $6^+$-face gives $\frac{1}{2}$ to every adjacent
	$3$-face that it shares one edge with, and it
	gives $1$ to every $3$-face that it shares at least two edges with.
	
\end{itemize}

Let $w'$ be the weight function obtained by applying (R1) 
and (R2) to the graph $H$
and the function $w$.
We shall prove that $w'(f) \geq 0$ for any face $f$ of $H$.

Now, since $G$ has no intersecting triangles
and every triangle is adjacent to at most one $4$-cycle,
each $3$-face in $H$ is adjacent to no $3$-face, and 
at most one $4$-face (via at most one edge). We consider some cases.
\begin{itemize}

\item If a $3$-face $f$ is adjacent to two $5^+$-faces
with no bad $3$-faces, then $w'(f) \geq -1 + 2 \times 1/2 =0$.

\item If a $3$-face $f$ is adjacent to a $5$-face with a bad $3$-face
via one edge, then the other two edges of $f$
do not lie on $4$-faces or on $5$-faces with bad $3$-faces; because
every triangle in $G$ is adjacent to at most one $4$-cycle.
Thus, we have  $w'(f) = -1 + 2 \times 1/2 =0$.

\item If a $3$-face $f$ is a bad $3$-face for an adjacent $5$-face,
then it receives $1$ from this $5$-face. Hence,
$w'(f) \geq -1 + 1 =0$.

\end{itemize}
In conclusion, every $3$-face $f$ satisfies that $w'(f) \geq 0$.

Every $4$-face $f$ clearly satisfies $w'(f) =0$. 
If a $5$-face $f$ is adjacent to a bad $3$-face then it
satisfies $w'(f) = 1-1=0$; if it is not adjacent to a bad $3$-face,
then it is adjacent to at most two distinct $3$-faces, since
no pair of triangles intersect in $G$; 
thus $w'(f) \geq 1- 2 \times \frac{1}{2} =0$.

Consider a $2k$-face $f$, where $k \geq 3$.
The face $f$ gives $1$ to an adjacent $3$-face that it shares at least
two edges with, and it gives $\frac{1}{2}$ to a $3$-face that it shares
one edge with.
Since no pair of triangles intersect, $f$
is adjacent to $3$-faces via at most $\frac{4k}{3}$ 
edges if $f$ shares at most two edges with every adjacent $3$-face.
Since a $3$-face which shares two or three edges with $f$
receives $1$ from $f$, it follows that
$w'(f) \geq 2k-4-\frac{2k}{3} \geq 0$
if $k \geq 3$.
A similar calculation shows that $w'(f) \geq 0$ if $f$ is a $(2k+1)$-face, 
where $k\geq 3$.

By \eqref{eq:euler}, we
have that $$\sum_{v \in V}w(v) + \sum_{f\in F}w(f) <0.$$
Hence, the same holds for the weight function $w'$ obtained by 
redistributing the charge.
Now, since $w'(f) \geq 0$ for any face $f$ in $H$, 
we have that $\sum_{v \in V} w'(v) < 0$, which means that
$\sum_{v \in V}(d(v)-4) < 0$, and so the average degree in $H$
is less than $4$, contrary to our assumption.
\end{proof}

Theorem \ref{main} and Lemma \ref{lem:orient} yield the following.

\begin{theorem}\label{IntesectingTriangles}
If $G$ is a planar graph with no intersecting
triangles and where every triangle is adjacent to at most
one $4$-cycle, then $\ch_{ad}(G) \leq 3$, and thus $G$ is $(3,1)$-choosable.

\end{theorem}

\begin{remark}
	It is easily verified that the proof of Theorem \ref{th:planar1}
		(with the exact same discharging rules and calculations)
	is valid 
	if the following two conditions hold:
	\begin{itemize}
	
		\item [(i)] every triangle satisfies that at most one of its edges
		lie on another cycle
		of length at most four;
					
		\item[(ii)] every $5$-, $6$-, and $7$-cycle 
		satisfies that at most 
		two, four, and six of its edges lie on triangles, respectively.

	\end{itemize}
	Hence, every planar graph $G$ satisfying these conditions is adaptably $3$-choosable, and also $(3,1)$-choosable.
\end{remark}
Another immediate consequence of the preceding theorem is the following.

\begin{corollary}
\label{cor:planar1}
If $G$ is a planar graph with no intersecting 
triangles and no intersecting $4$-cycles,
then $G$ is adaptably $3$-choosable, and thus $(3,1)$-choosable.
\end{corollary}

Using a similar argument as in the proof of Theorem \ref{th:planar1}
we can prove the following.

\begin{theorem}
\label{th:planar2}
		If $G$ is a planar graph where
		no triangle of $G$ is adjacent to a triangle or a $4$-cycle, 
		and each $5$-cycle is adjacent to at most three
		triangles, then $\mad(G) < 4$.
		Hence, $G$
		is adaptably $3$-choosable and $(3,1)$-choosable.
\end{theorem}
\begin{proof}
	The proof of Theorem \ref{th:planar2} is similar to the proof of
	Theorem \ref{th:planar1}. The only substantial difference
	is that instead of the discharging rules (R1) and (R2) we apply the following.
	
	\begin{itemize}

	\item [(R3)] For each edge between a
	$5^+$-face and a triangle, the $5^+$-face
	gives $\frac{1}{3}$ to the adjacent triangle.
\end{itemize}
Similar calculations as in the proof of 
Theorem \ref{th:planar1}
then yield the desired contradiction; the details are
omitted.
\end{proof}

We have to notice that our Theorem \ref{IntesectingTriangles} gives a very short proof of  the following theorem proved in  \cite{GuanZhu}.
\begin{theorem}
Suppose G is a planar graph. Then G is adaptably 3-choosable if any two triangles in G have distance at least 2 and
no triangle is adjacent to a 4-cycle.
\end{theorem}

	\section{Halin graphs}
	
	A {\em Halin graph} is a planar graph constructed from a planar drawing
	of a tree with at least four vertices and with no vertices
	of degree two 
	by connecting its leaves by a cycle that crosses none of its edges.
	The following lemma is well-known and easy to prove.
	
	\begin{lemma}
	\label{lem:cycle}
		Every cycle is $(2,1)$-choosable.
	\end{lemma}
	
	\begin{proposition}
	\label{prop:Halin}
	Every Halin graph is $(3,1)$-choosable.
	\end{proposition}
	\begin{proof}
		Let $G = T \cup C$ be a Halin graph, where $T$ is the spanning tree, and
		$C$ is the outer cycle. Consider a $(3,1)$-list assignment $L$ for $G$.
		Now, any tree is trivially $(2,1)$-choosable; 
		thus, we can pick an $L$-coloring
		$\varphi$ of the tree $T'$ obtained from $T$ by removing all leaves. 
		Now, for every vertex $v$ of $C$ we define a new list assignment
		$L'$ by removing any color from $L(v)$ 
		that is used on a neighbor of $v$ in $T'$.
		Note that $L'$ is $(2,1)$-list assignment for $C$. By Lemma \ref{lem:cycle},
		$C$ is $L'$-colorable. By taking an $L'$-coloring of $C$ together with the
		coloring $\varphi$ of $T'$, we obtain an $L$-coloring of $G$.
	\end{proof}
	
	Since $K_4$ is $(3,1)$-choosable, but not $(2,1)$-choosable, Proposition
	\ref{prop:Halin} is in fact sharp.
	
	\bigskip

\section{Acknowledgement}
Carl Johan Casselgren was supported by a grant from the Swedish
Research Council (2017-05077).\\
Andr\'e Raspaud was partially supported by the ANR project HOSIGRA (ANR-17-CE40-0022).

\end{document}